\documentclass[journal,twoside,web]{ieeecolor}

\usepackage{generic}
\usepackage{lcsys}

\usepackage{amsmath,amssymb,amsfonts,mathtools}
\usepackage{mathtools, cuted}
\usepackage{blindtext}
\usepackage{graphicx}
\usepackage{hyperref}
\usepackage{pgfplots}
\usepackage{pgfplotstable}
\usepgfplotslibrary{groupplots,fillbetween}
\usetikzlibrary{calc}
\pgfplotsset{compat=1.18}
\usetikzlibrary{plotmarks}

\definecolor{mblue}{HTML}{005AB5}
\definecolor{mred}{HTML}{D55E00}
\definecolor{mpink}{HTML}{7B3294}
\definecolor{mcyan}{HTML}{DDAA33}
\definecolor{onlineblue}{RGB}{0,114,178}
\definecolor{valueorange}{RGB}{213,94,0}
\definecolor{unboundedgray}{RGB}{232,232,232}
\definecolor{gridgray}{RGB}{214,214,214}
\definecolor{mplpink}{HTML}{7B3294} 
\definecolor{mplred}{HTML}{D55E00} 
\definecolor{mplblue}{HTML}{005AB5} 
\definecolor{mplcyan}{HTML}{DDAA33}
\definecolor{mplgray}{HTML}{111111}     

\newtheorem{thm}{Theorem}

\newtheorem{prop}{Proposition}

\newcommand{\F}{\mathcal{F}}
\newcommand{\Sc}{\mathcal{S}}
\newcommand{\Xb}{\mathbf{X}}
\newcommand{\Ub}{\mathbf{U}}
\newcommand{\Zb}{\mathbf{Z}}

\DeclareMathOperator{\vect}{vec}
\DeclareMathOperator{\mat}{mat}
\DeclareMathOperator{\diag}{diag}

\newcommand{\review}{\color{black}}

\def\BibTeX{{\rm B\kern-.05em{\sc i\kern-.025em b}\kern-.08em
    T\kern-.1667em\lower.7ex\hbox{E}\kern-.125emX}}
\begin{document}

\title{Bounded Linear Programs for Data-Driven Optimal Control via Moment-Matching}

\author{Andrea Martinelli, Lucia Pezzetti, Niklas Schmid, Florian D\"orfler and John Lygeros
\thanks{Research supported by the European Research Council under the Horizon 2020 Advanced Grant No. 787845 (OCAL) and the ETH AI Center.}
\thanks{The authors are with the Automatic Control Laboratory, ETH Z\"urich, Physikstrasse 3, 8092 Zurich, Switzerland. E-mails: \tt\footnotesize \{andremar, lpezzetti,nikschmid,dorfler,lygeros\}@ethz.ch.}%
}

\maketitle
\thispagestyle{empty}

\begin{abstract}

Linear programming (LP) formulations offer a conceptually elegant approach to infinite-horizon, model-free nonlinear optimal control in continuous spaces. However, in addition to the curse of dimensionality, their practical use is limited by the difficulty of consistently obtaining bounded solutions.
In this work, we use moment-matching techniques to derive sufficient boundedness conditions in terms of the available dataset and the cost vector of the LP. Moreover, we discuss practical design methods for nonlinear systems and polynomial features.
\end{abstract}



\section{INTRODUCTION} \label{Sec:intro}

\IEEEPARstart{T}{he} linear programming (LP) method is a fundamental approach to solving optimal control problems \cite{LasserreDTMCP,BertsekasVol2}. Introduced by Manne in the 1960s \cite{ManneLP}, the method is based on the idea that the solution to the Bellman equation can be cast as the optimizer of an LP. Like all dynamic programming methods, the LP approach suffers from the \textit{curse of dimensionality} \cite{BellmanDP,bertsekas2019RL}. The exact LP formulation for finite state-action spaces was primarily viewed as a theoretical result for many years, since it requires formulating an LP where the number of constraints is equal to the number of state-action pairs \cite[Ch. 3.8]{PowellADP}. The method has seen a resurgence in recent years, in conjunction with increasing compute, major advances in LP solution methods \cite{BoydConvexOptimization}, and especially through approximation schemes to deal with continuous state-action spaces and model-free approaches \cite{SutterCDC2017,PaulADP,WangIteratedBellInequalities,SummersSumofSquares}.

Many real-world applications require operating a system without full knowledge of its dynamics \cite{SuttonRLanIntroduction}. Learning how to optimally control a system is possible by letting it interact with the environment, collecting information about state transitions and costs over time, and using the resulting data to build a sampled version of the LP \cite{RelaxedBellmanOp, MartinelliAffine, GoranADP}. This data-driven version of the LP is typically formulated in terms of the $Q$-function \cite{WatkinsQLearning,CogillDecentralizedADP}, so that the associated control policy can be computed in a model-free fashion. Most of the recent literature has shown encouraging results, particularly for low-dimensional systems. For higher-dimensional problems, the key obstacle is the lack of systematic boundedness guarantees for the approximate LP.


As showcased by the numerical results in \cite{RelaxedBellmanOp}, the number of constraints needed to keep the LP bounded grows dramatically with the dimension of the system. Most existing examples therefore rely on low-dimensional systems, large datasets, finite state-action spaces, or explicit bounds/regularizers. {\review Regularizers are proposed in \cite{SutterCDC2017,deFariasConstraintSampling,petrik2010feature_selection}, but a drawback is that they can introduce a significant distortion to the optimizers. Instead of regularizing, the authors in \cite{lakshminarayanan2018linearly_relaxed_alp} and \cite{shariff2020efficient_planning} propose to select and combine specific constraints so as to preserve suitable geometric structure. Unfortunately, this idea cannot be directly applied in a model-free and continuous states setting, since the constraints come from exploration and their combination typically results in non-admissible ones \cite{AndreaSynthesisBellmanIneq}. 

To the best of our knowledge, the only works to consider boundedness conditions in model-free and continuous states setting, without using regularizers, are \cite{pazis2011nonparametric_alp}, \cite{CDC22_meyn} and \cite{CDC23_meyn_stoch}. The authors in \cite{pazis2011nonparametric_alp} propose a non-parametric version of the LP with an additional Lipschitz bound which, on one hand, alleviates the boundedness problem but, on the other, requires a number of constraints that scales quadratically with the number of samples. In \cite{CDC22_meyn} and \cite{CDC23_meyn_stoch}, though in finite action spaces, they provide a sufficient condition based on the rank of a data-based covariance matrix. Although providing useful analysis on data excitation, this condition does not come with an explicit method to design the LP parameters or to sample the constraints accordingly. The authors in \cite{LuciaBounded} tackle the problem for the class of LPs associated to a policy, that is, the LPs whose solution is the cost associated to a fixed policy. {\review Inspired by \cite{LuciaBounded},} we are interested in boundedness conditions for the \textit{optimal} LP formulation, that is, the LP whose solution is the optimal $Q$-function. Our contributions are focused on model-free and continuous state-action settings, and summarized as follows.}
\begin{itemize}
    \item[i.] We explicitly describe the cone of directions for which the LP is bounded as the conic hull of a data matrix which depends on the choice of basis functions;
    \item[ii.] In the case of polynomial basis functions, we propose a \textit{moment-matching} method to design the cost vector based on the observed data;
    \item[iii.] We provide high-dimensional simulations with minimal data to ensure that, besides guaranteeing finite solutions, our method retains {\review satisfactory} performance both for linear and nonlinear systems.
\end{itemize}

\section{EXACT LINEAR PROGRAMS}

Consider a discrete-time dynamical system 
\begin{equation}
 x^+ = f(x,u),    
\end{equation}
with possibly infinite state and action spaces $x \in \mathbf{X} \subseteq \mathbb{R}^{n}$ and $u \in \mathbf{U} \subseteq \mathbb{R}^{m}$, where $f : \mathbf{X} \times \mathbf{U} \rightarrow \mathbf{X}$ is the map encoding the dynamics.	
We work with \textit{stationary feedback policies}, given by functions $\pi : \mathbf{X} \rightarrow \mathbf{U}$. A nonnegative cost is associated to each state-action pair through the \textit{stage cost} function $\ell : \mathbf{X} \times \mathbf{U} \rightarrow \mathbb{R}_{+}$. We introduce a \textit{discount factor} $\gamma \in (0,1)$ and consider the infinite-horizon cost associated to a policy,
	$v_{\pi}(x) = \sum_{k=0}^{\infty}\gamma^{k}\ell(x_k,\pi(x_k))$, $x_0 = x$.
The objective of the optimal control problem is to find an optimal policy $\pi^{\star}$ such that $v_{\pi^{\star}}(x) = \inf_{\pi} v_{\pi}(x) = v^{\star}(x)$, where $v^\star$ is known as the optimal \textit{value function}. {\review We work under Assumptions 4.2.1 and 4.2.2 in \cite{LasserreDTMCP} to ensure the problem is well-posed.} For the sake of notation, we denote the state-action pair $z=(x,u)$ and $\Zb = \Xb \times \Ub$. Moreover, we denote by $\mathcal{S}(\Zb)$ the vector space of real-valued measurable functions with finite weighted sup-norm \cite[Lemma 2.20]{LasserreDTMCP}, and by $\mathcal{M}_+(\Zb)$ the set of finite non-negative measures on $\Zb$ with finite weighted total variation such that $c(\Zb)>0$. 

In the context of model-free control, working directly with value functions can be difficult because extracting $\pi^\star$ from $v^\star$ is not possible if the dynamics $f$ or the stage cost $\ell$ are not known \cite{GoranADP}. This difficulty can be addressed by introducing the $Q$-function \cite{WatkinsQLearning} associated to a policy $\pi$ as
\begin{align*}
	q_{\pi}(z) & = \ell(z) + \gamma v_{\pi}(f(z)) 
	= \ell(z) + \gamma q_{\pi}(f(z),\pi(f(z))).
\end{align*}
This can be interpreted as the cost of applying control input $u$ at state $x$, and following policy $\pi$ thereafter. The optimal $Q$-function is expressed by
\begin{align*}
	q^{\star}(z) & = \underbrace{\ell(z) + \gamma \inf_{w\in \mathbf{U}}q_{\pi^\star}(f(z),w)}_{\mathcal{F}q^{\star}(z)},
\end{align*}
where the operator $\F : \mathcal{S}(\Zb) \to \mathcal{S}(\Zb)$ is the \textit{Bellman operator} for $Q$-functions, it is \textit{monotone} and \textit{$\gamma$-contractive} \cite{BertsekasVol2}, and therefore admits  a unique fixed point $q^\star$.
The link between $v^\star$ and $q^\star$ is then given by $v^\star(x) = \inf_{u\in \mathbf{U}}q^\star(x,u)$.
The advantage of the $Q$-function reformulation is that the computation of the policy does not require $f$ or $\ell$, as
\begin{equation} \label{policy_extraction}
	\pi^\star(x) = \arg \min_{u\in\mathbf{U}}q^\star(x,u).
\end{equation}

Focusing now on $\F$, we observe that the Bellman inequality $q \le \F q$ implies $q \le q^{\star}$. It is then natural to look for the greatest $q \in \mathcal{S}(\Zb)$ that satisfies $q \le \F q$,
\begin{equation}\label{NonlinearProgramQfunctions}
	\begin{aligned}
		\sup_{q\in\mathcal{S}} \;\; & \int_{\Zb} q(z)c(dz) \\
		\mbox{s.t.} \;\; & q(z) \le \mathcal{F}q(z) \quad \forall z \in \Zb.
	\end{aligned}
\end{equation}
Then, if $c\in\mathcal{M}_+(\Zb)$, the solution of \eqref{NonlinearProgramQfunctions} coincides with $q^\star$ for $c$-almost all $z \in \Zb$ \cite{LasserreDTMCP,GoranADP}. Although $\F$ is a nonlinear operator, as discussed for instance in \cite{PaulADP} and \cite{GoranADP}, it is possible to reformulate \eqref{NonlinearProgramQfunctions} as an equivalent LP by dropping the infimum in $\F$, obtaining
\begin{equation}\label{LPdeterministic}
	\begin{aligned}
		\sup_{q\in \mathcal{S}} \;\; & \int_{\Zb} q(z)c(dz) \\
		\mbox{s.t.} \;\; & q(z) \le \ell(z) + \gamma q(f(z),w) \quad \forall(z,w) \in \Zb\times\mathbf{U}.
	\end{aligned}
\end{equation}
Because of the relaxation of the constraint set, the feasible region now includes an auxiliary input $w\in\mathbb{R}^m$. In general, however, the obtained formulation is infinite dimensional, and it is not solvable directly due to several sources of intractability which are collectively referred to as the \textit{curse of dimensionality}, see e.g. \cite{PaulADP} and \cite{WangIteratedBellInequalities}. {\review Note that, in case of stochastic systems, reformulating the LP as in~\eqref{LPdeterministic} is not trivial \cite{CogillDecentralizedADP, RelaxedBellmanOp}.}

\section{DATA-DRIVEN LINEAR PROGRAMS}

\subsection{Approximate Linear Program from Data}

Computing the $Q$-function by solving the LP \eqref{LPdeterministic} is in general computationally intractable because i) the optimization variable $q(z)$ lies in the infinite dimensional space $\Sc (\Zb)$, and ii) the number of constraints is infinite since the inequalities need to be satisfied for all $(z,w) \in \Zb \times \Ub$. To overcome these sources of intractability,  the original infinite problem is approximated by a finite one \cite{CogillDecentralizedADP, PaulADP, deFariasConstraintSampling}. 

As customary in the LP framework (see e.g. \cite{deFariasLPapproach}), one can restrict $q$ to the span of a finite family of basis functions $\phi_i : \Zb \to \mathbb{R}$ for $i\in \{1,\ldots,r\}$ as
\begin{equation} \label{basis_functions}
	\Sc^{\phi} (\Zb) =  \left \{ q\in\Sc : q(z)=\alpha^\top \phi (z) \right\},
\end{equation} 
where $\phi(z) = \begin{bmatrix} \phi_1(z) & \cdots & \phi_r(z) \end{bmatrix}^\top \in \mathbb{R}^r$, and $\alpha \in \mathbb{R}^r$ takes the role of the finite-dimensional optimization variable. Indeed, substituting $ \Sc $ with $ \Sc^{\phi}$ in \eqref{LPdeterministic} leads to a semi-infinite approximation of the original problem.
Important classes of approximation for $Q$-functions are based on radial basis functions and polynomials \cite{SuttonRLanIntroduction, BertsekasVol2, SummersSumofSquares}. To make sure that the calculation of the associated policy \eqref{policy_extraction} is tractable, $q(x,u)$ is often selected to be convex in $u$.

For the infinite number of constraints, if the dynamics $f$ and the cost $\ell$ are available or have a special structure (\textit{e.g.}, quadratic functions), and the features vector $\phi$ is selected appropriately, the constraint set can sometimes be represented exactly or relaxed via the $\Sc$-procedure or sum-of-squares \cite{WangIteratedBellInequalities, SummersSumofSquares}.
However, if the dynamics are unknown, these tools cannot be used to approximate the feasible set. 
In this case, the constraints can be relaxed via sampling \cite{deFariasConstraintSampling}.
If we measure the current state $x_i \in \Xb$, apply the input $u_i \in \Ub$, observe the transition $x_i^+ = f(x_i, u_i)$ and the incurred cost $\ell_i=\ell(x_i, u_i)$, we obtain a set of data tuples $(x_i,u_i,x_i^+,\ell_i)$. Note that we use index $i$ instead of $t$ to emphasize that data tuples do not have to be consecutive in time; for convenience, we denote $z_i=(x_i,u_i)$ and $z_i^+=(x_i^+,w_i)$. In the spirit of, \textit{e.g.}, \cite{GoranADP} and \cite{RelaxedBellmanOp}, the resulting LP that approximates $q^\star$ can be interpreted as solving a reinforcement learning problem, where the goal is to learn a policy without knowing the dynamics and cost function. 

By substituting $\mathcal{S}$ with $\mathcal{S}^\phi$ and sampling $N$ constraints, we obtain the following data-driven finite-dimensional approximation of \eqref{LPdeterministic},
\begin{equation} \label{data-driven_LP} 
\begin{aligned} 
	 \max_{\alpha \in \mathbb{R}^r}  \;\; & \alpha^\top \int_{\Zb} \phi (z) c(dz) \\
	 \text{s.t.} \;\; & \bold\Phi ^\top \alpha \leq \boldsymbol\ell\,,
\end{aligned} 
\end{equation} 
where $\bold\Phi=\begin{bmatrix} \phi(z_1)-\gamma\phi(z_1^+) & \cdots & \phi(z_N)-\gamma\phi(z_N^+) \end{bmatrix}$ is an $r\times N$ data matrix containing information on the observed trajectories of the system, and $\boldsymbol\ell=\begin{bmatrix}
    \ell_1 & \cdots & \ell_N \end{bmatrix}^\top \in \mathbb{R}^N$ collects the incurred costs.
    
\subsection{Practical Considerations and Difficulties}

A relevant consideration is that, while for any $c\in\mathcal{M}_+$ solving problem \eqref{LPdeterministic} yields $q^\star$ $c$-almost everywhere, this is no longer the case for the approximate LP \eqref{data-driven_LP}. In this case, the choice of $c$ may have a significant impact on the quality of the resulting approximation. The choice of $c$ can be interpreted as allocating approximation quality over the state-action space \cite{deFariasLPapproach, PaulADP}. In the LP literature, $c$ is typically selected to be a {\review fixed} probability measure, \textit{e.g.}, if the state-action space is unbounded one can use a Gaussian distribution, or if it is compact a uniform distribution {\review \cite{deFariasLPapproach,CDC22_meyn,pazis2011nonparametric_alp, CogillDecentralizedADP}}. The authors in \cite{PaulADP} propose solving the approximate LP problem for multiple realizations of $c$ and then computing a point-wise maximum over all derived solutions. However, in general, there are no guidelines on how to select $c$. 

Arguably the most important issue related to the approximate LP, that perhaps has hampered the deployment of this method, is the difficulty of obtaining bounded solutions, especially as the state and input dimensions grow \cite{RelaxedBellmanOp,LuciaBounded}.
Under ergodicity assumptions, the authors in \cite{CDC22_meyn} state in Theorem 2.2 that, if the covariance matrix associated with the basis functions $\Sigma \coloneq  ( \lim_{N \to \infty} \frac{1}{N} \sum_{i=1}^{N} \phi(z_i) \phi(z_i)^\top ) - ( \lim_{N \to \infty} \frac{1}{N} \sum_{i=1}^{N} \phi(z_i) ) ( \lim_{N \to \infty} \frac{1}{N} \sum_{i=1}^{N} \phi(z_i) )^\top$ is full rank (corresponding to a persistence-of-excitation condition), then the constraint region of \eqref{data-driven_LP} is bounded for $N$ sufficiently large.
While this result provides useful insight into the role of data richness, it does not directly provide a practical procedure to guarantee boundedness in finite-sample settings. Indeed, {\review since the number of constraints $N$ needed to obtain a bounded solution quickly becomes prohibitive as the state-action space grows, the constraint region may remain unbounded for any reasonable amount of constraints even when $\Sigma$ is full rank.}

Our goal is to derive explicit conditions that guarantee a bounded optimal value of the LP based on the observed data. Importantly, these conditions do not require the feasible region itself to be bounded; instead, they characterize objective directions for which the LP admits a finite solution even when the feasible set is unbounded. In the next section, we complement the results in \cite{CDC22_meyn} by providing a systematic method to design the measure $c$ from data to ensure this property. 

\section{BOUNDEDNESS GUARANTEES}

\subsection{General Features}

The following result describes the cone of directions for which the data-driven LP is bounded.

\begin{thm} \label{mainthm}
	The solution to \eqref{data-driven_LP} is finite if and only if there exists $\lambda\in\mathbb{R}^{N}_+$ such that
	\begin{equation} \label{bounded_condition}
	\int_{\Zb} \phi (z) c(dz) = \bold\Phi \lambda\,.
	\end{equation}
\end{thm}
\vspace{0.4cm}
\begin{proof}
	The dual LP associated to \eqref{data-driven_LP} is \cite[Ch.6]{matouvsek2007understanding} 
	\begin{equation} \label{dual_LP}
		\begin{aligned} 
			\min_{\lambda \in \mathbb{R}^N} \;\; & \sum_{i=1}^{N} \lambda_i \ell(z_i) \\
			\text{s.t.} \;\; & \bold\Phi \lambda = \int_{\Zb} \phi (z) c(dz), \quad \lambda \ge 0.
		\end{aligned} 
	\end{equation} 
Note that $\alpha=0$ is always a feasible solution to the primal LP \eqref{data-driven_LP}. Hence, by duality theory \cite[Ch.6]{matouvsek2007understanding}, the primal LP \eqref{data-driven_LP} is bounded if and only if the dual LP \eqref{dual_LP} is feasible.
\end{proof}

Condition \eqref{bounded_condition} has an important geometric interpretation in terms of the direction of growth of the LP: \eqref{data-driven_LP} is bounded if and only if the cost vector $\int_{\Zb} \phi (z) c(dz)$ lies in the conic hull of the columns of the data matrix $\bold\Phi$. Since $c$ is typically selected to be a fixed probability distribution, once $c$ and the basis functions $\phi$ are designed, then also the vector $\int_{\Zb} \phi (z) c(dz)$ is fixed. In principle, adding more and more data increases the number of columns of $\bold\Phi$, and could eventually lead to \eqref{bounded_condition} being satisfied. Unfortunately, in practice, this condition is typically difficult to satisfy even for large datasets. {\review The solution proposed in \cite{LuciaBounded} for the LP associated to the $Q$-function of a fixed policy in the LQR setting involves designing $c$ to satisfy \eqref{bounded_condition}. The context of policy evaluation and LQR design considered in \cite{LuciaBounded} allows for a simplified analysis, since the Bellman operator is linear in $q$ and the constraint set is a cone. In the next section, we extend this approach to the LP class associated to the optimal $Q$-function \eqref{NonlinearProgramQfunctions} (and its data-driven approximation \eqref{data-driven_LP}) for general nonlinear systems and polynomial features.}

\subsection{Polynomial Features}

Polynomial basis functions are a common choice to approximate the value/$Q$-function, offering a compromise between approximation quality and parametrization complexity \cite{SuttonRLanIntroduction, SummersSumofSquares}. Note that any polynomial of degree $\le 2d$ can be represented as a quadratic form of monomials \cite{LasserrePoly}, as 
\begin{equation*}
    q(z) = p(z)^\top Qp(z)\,,
\end{equation*}
where $p(z)$ is the complete vector of monomials in $z$ of degree at most $d$, and $Q$ is a symmetric matrix. By using the fact that $p(z)^\top Q p(z) = (\vect Q)^\top \vect(p(z)p(z)^\top)$, we can define the space of polynomial features according to the notation introduced in \eqref{basis_functions} as
\begin{equation*}
	\begin{aligned}
		\Sc^{\text{poly}} (\Zb) =  \{ q\in\Sc : q(z) = (\underbrace{\vect Q}_{\alpha})^\top \underbrace{\vect(p(z)p(z)^\top)}_{\phi(z)} \}.
	\end{aligned}
\end{equation*} 
Note that $p(z) \in \mathbb{R}^\rho$, where $\rho=\binom{n+m+d}{d}$, and because $Q$ is symmetric, the number of parameters needed to represent the $Q$-function is $r=\rho(\rho+1)/2$.

Specializing \eqref{bounded_condition} to the polynomial case, \textit{i.e.}, imposing $\phi(z)=\vect(p(z)p(z)^\top)$, leads to the following result. 
\begin{prop}\label{prop1}
    If $q\in\Sc^{\text{poly}}$, then the solution to \eqref{data-driven_LP} is finite if and only if there exists $\lambda\in\mathbb{R}^{N}_+$ such that
	\begin{equation} \label{nec_condition}
    \int_{\Zb} p(z)p(z)^\top c(dz) = \mat (\bold \Phi \lambda)\,, 
	\end{equation}
    where $\mat (\bold \Phi \lambda)=\sum_{i=1}^{N} \lambda_i (p(z_i)p(z_i)^\top -\gamma p(z_i^+)p(z_i^+)^\top)$
\end{prop}
is the matricial version of $\bold \Phi \lambda$.

For a fixed $\lambda\ge0$, Proposition \ref{prop1} requires the existence of a non-negative measure $c$ that satisfies \eqref{nec_condition}. This is an established problem in measure theory, known as the \textit{moment problem,}\footnote{Given a sequence of moments, the moment problem seeks to establish conditions under which there exists a measure with that exact moment sequence. Equivalently, given a real square matrix, conditions under which such matrix is effectively a moment matrix associated to a measure \cite{LasserrePoly}. \review{Numerous problems can be viewed as particular instances of the moment problem, resulting in successful applications in the areas of optimization, statistics, and control \cite[Part II]{LasserrePoly}}.} where the term $\int_{\Zb} p(z)p(z)^\top c(dz)$ is precisely the definition of \textit{moment matrix} associated to the measure $c$ \cite[Sec. 3.2]{LasserrePoly}. Establishing whether a given matrix $\mat(\bold\Phi \lambda)$ is a moment matrix for some measure $c$ is difficult, in general. A necessary condition is that $\mat(\bold\Phi \lambda)\succeq0$ but, to the best of the authors' knowledge, there is no tractable necessary and sufficient condition that solves the truncated and multivariate moment problem under study here. The only case for which $\mat(\bold\Phi \lambda)\succeq0$ is also sufficient is for polynomials of degree at most $2$ \cite{CurtoMomentProblem}, denoted here with $\Sc^{\text{quad}}\subset\Sc^{\text{poly}}$. In this case, the moment problem reduces to matching mean and covariance, which always admits a probability measure when the covariance matrix is positive semi-definite. 
\begin{prop}\label{prop2}
    If $q\in\Sc^{\text{quad}}$, then the solution to \eqref{data-driven_LP} is finite if and only if there exists $\lambda\in\mathbb{R}^{N}_+\setminus \{0\}$ such that
	\begin{equation} \label{iffcond}
    \mat (\bold \Phi \lambda) \succeq 0\,. 
	\end{equation}
    In that case, \eqref{data-driven_LP} is guaranteed to be bounded by selecting $\int_{\Zb} \phi (z) c(dz) = \bold \Phi \lambda$ as the cost vector.
\end{prop}

For quadratic features, Proposition \ref{prop2} fully characterizes the cone of directions for which \eqref{data-driven_LP} is bounded, in terms of the data $\bold \Phi$ and without explicitly computing the measure $c$. More generally, for polynomials of degree $\le d$, let us introduce an $r\times M$ auxiliary data matrix $\bold\Phi_A=\begin{bmatrix}
    \phi(y_1) & \cdots & \phi(y_M)
\end{bmatrix}$, where $y_1,\ldots,y_M\in\Zb$ are vectors in the state-action space sampled according to a given distribution. 
\begin{prop} \label{proppoly}
    If $q\in\Sc^{\text{poly}}$, then the solution to \eqref{data-driven_LP} is finite if there exist $\lambda\in\mathbb{R}^{N}_+\setminus \{0\}$ and $\mu\in\mathbb{R}^{M}_+\setminus \{0\}$ such that
    \begin{equation} \label{ifcond}
        \mat (\bold \Phi \lambda) = \mat (\bold \Phi_A \mu)\,,
    \end{equation}
    where $\mat (\bold \Phi_A \mu)=\sum_{i=1}^{M} \mu_i p(y_i)p(y_i)^\top$.
    In that case, \eqref{data-driven_LP} is guaranteed to be bounded by selecting $\int_{\Zb} \phi (z) c(dz) = \bold \Phi \lambda$ as the cost vector (equivalently, $\bold \Phi_A \mu$).
\end{prop}
\begin{proof}
    Consider a pair $\lambda,\mu$ for which \eqref{ifcond} holds. Then, \eqref{nec_condition} reads $\int_{\Zb} p(z)p(z)^\top c(dz)=\sum_{i=1}^{M} \mu_i p(y_i)p(y_i)^\top$ which is satisfied, for instance, by the measure $c=\sum_{i=1}^M \mu_i \delta_{y_i}$, where $\delta_{y_i}$ is the Dirac measure at $y_i$.
\end{proof}
Proposition \ref{proppoly} allows one to compute cost vectors that lead to finite solutions for \eqref{data-driven_LP} in the case of polynomial features, using only the observed data $\bold \Phi$ and without {\review having to perform the high-dimensional integration in \eqref{nec_condition}.} We refer to the computation of the cost vector in \eqref{data-driven_LP} via \eqref{ifcond} as \textit{moment-matching}. Although the condition is only sufficient, we show that, in practice, it consistently leads to a finite solution, even for large-scale systems with minimal data. Moreover, note that \eqref{ifcond} is a stronger requirement than \eqref{iffcond} since $\mat (\bold \Phi_A \mu)\succeq0$ by construction. 

The measure $c$ can be interpreted as weighting the approximation quality over the state-action space \cite{PaulADP}. In this sense, selecting the cost vector for \eqref{data-driven_LP} based on data according to Proposition \ref{proppoly} can be interpreted as selecting $c=\sum_{i=1}^M \mu_i \delta_{y_i}$, weighting approximation quality at the points $y_1,\ldots,y_M$ with relative weights $\mu_1,\ldots,\mu_M$. {\review In this sense, the choice of auxiliary points determines where the approximation quality is distributed over the state-action space, while the observed data allocate relative weights via \eqref{iffcond} and \eqref{ifcond}.}

Next, consider $\Sc^{\text{poly}}_{u^2}$, the class of polynomial functions $q(x,u)=\begin{bsmallmatrix}
    p(x) \\ u 
\end{bsmallmatrix}^\top Q \begin{bsmallmatrix}
    p(x) \\ u 
\end{bsmallmatrix}$ of degree up to $2d$ in $x$ and up to 2 in $u$, with
$Q = \begin{bsmallmatrix} Q_{xx} & Q_{xu} \\ Q_{xu}^\top  & Q_{uu}
\end{bsmallmatrix}$. 
Clearly,  $\Sc^{\text{quad}}\subset\Sc^{\text{poly}}_{u^2}\subset\Sc^{\text{poly}}$.
\begin{prop}
    If $q\in\Sc^{\text{poly}}_{u^2}$, and the optimal solution to LP \eqref{data-driven_LP} is such that $Q_{uu}\succ0$, then the optimal greedy policy \eqref{policy_extraction} is unique and has the closed-form expression
    \begin{equation} \label{closedformpolicy}
        \pi(x)={\review -}Q_{uu}^{-1}Q_{xu}^\top p(x).
    \end{equation}
\end{prop}
\vspace{0.15cm}
\begin{proof}
 Denote by $q(x,u)=\begin{bsmallmatrix}
    p(x) \\ u 
\end{bsmallmatrix}^\top Q \begin{bsmallmatrix}
    p(x) \\ u 
\end{bsmallmatrix}$ the solution to \eqref{data-driven_LP}. 
Since $Q_{uu}\succ0$ and thus $q$ is convex in $u$, the associated control policies $u=\pi(x)$ defined in \eqref{policy_extraction} are the solution set to $\tfrac{\partial q(x,u)}{\partial u}=0$, which is $Q_{uu}u + Q_{xu}^\top p(x)=0$. Because $Q_{uu}$ is invertible, the unique solution to the previous equation is \eqref{closedformpolicy}, concluding the proof.
\end{proof}

Although $\Sc^{\text{poly}}_{u^2}$ is a more restrictive choice than  $\Sc^{\text{poly}}$, it has the advantage of having a closed-form expression for the associated policy. {\review Before computing \eqref{closedformpolicy}, one has to check that $Q_{uu}\succ0$.} As a final note, for quadratic basis functions $q(x,u)= \begin{bsmallmatrix}
    x \\ u
\end{bsmallmatrix}^\top Q \begin{bsmallmatrix}
    x \\ u
\end{bsmallmatrix}$, one has $p(x)=x$, reducing \eqref{closedformpolicy} to the familiar class of linear policies $\pi(x)={\review -}Q_{uu}^{-1} Q_{xu}^\top x$. 

\section{NUMERICAL EXAMPLES}

\begin{figure*}[!t]
    \centering
    \begin{minipage}[t]{0.38\textwidth}
    \centering
\begin{tikzpicture}
\begin{axis}[
    width=\dimexpr\linewidth-1.15cm\relax,
    height=3.2cm,
    scale only axis,
    log basis x=10,
    xmin=230, xmax=2050,
    ymin=0, ymax=31.5,
    xlabel={Number of samples (N)},
    ylabel={State dimension (n)},
    title={\textbf{Boundedness Isolines}},
    xlabel style={font = \small, yshift=3pt},
    ylabel style={font = \small, yshift=-3pt},
    title style={font = \small, yshift=-5.3pt},
    tick label style={font = \scriptsize},
    xtick={250,500,750,1000,1250,1500,1750,2000},
    xticklabels={250,500,750,1000,1250,1500,1750,2000},
    ytick={2,3,...,30},
    yticklabel={%
        \pgfmathtruncatemacro{\ytickint}{\tick}%
            \ifodd\ytickint
        \else
            \ytickint
        \fi
    },
    minor x tick num=0,
    axis line style={black},
    tick style={black},
    grid=both,
    major grid style={gray!45, densely dotted, line width=0.15pt},
    minor grid style={gray!25, densely dotted, line width=0.15pt},
    legend style={
        at={(0.98,0.5)},
        anchor=east,
        draw=gray!35,
        fill=white,
        fill opacity=0.82,
        text opacity=1,
        font=\scriptsize,
        inner sep=0pt,
        legend cell align=left,
        /tikz/every even column/.append style={column sep=0.05cm},
    },
    legend columns=2,
    legend image code/.code={
        \draw[
            #1,
            mark repeat=2,
            mark phase=2,
            mark options={solid}
        ]
        plot coordinates {
            (0cm,0cm)
            (0.13cm,0cm)
            (0.26cm,0cm)
        };
    },
    clip mode=individual,
]

\addplot[draw=none, fill=mblue, fill opacity=0.12, forget plot]
coordinates {(250,0) (250,17) (500,20) (750,22) (1000,24) (1250,25) (1500,26) (1750,26) (2000,26) (2000,0)} -- cycle;
\addplot[draw=none, fill=mred, fill opacity=0.12, forget plot]
coordinates {(250,0) (250,19) (500,23) (750,25) (1000,27) (1250,29) (1500,29) (1750,30) (2000,30) (2000,0)} -- cycle;
\addplot[draw=none, fill=mpink, fill opacity=0.12, forget plot]
coordinates {(250,0) (250,21) (500,26) (750,29) (1000,30) (1250,30) (1500,30) (1750,30) (2000,30) (2000,0)} -- cycle;
\addplot[draw=none, fill=mcyan, fill opacity=0.12, forget plot]
coordinates {(250,0) (250,24) (500,30) (750,30) (1000,30) (1250,30) (1500,30) (1750,30) (2000,30) (2000,0)} -- cycle;
\addplot[draw=none, fill=black, fill opacity=0.12, forget plot]
coordinates {(250,0) (250,2) (500,2) (750,3) (1000,4) (1250,3) (1500,5) (1750,4) (2000,5) (2000,0)} -- cycle;

\addplot+[mblue, thick, solid, mark=*, mark size=1.8pt, mark options={draw=mblue, fill=mblue}]
coordinates {(250,17) (500,20) (750,22) (1000,24) (1250,25) (1500,26) (1750,26) (2000,26)};
\addlegendentry{$M$=250}

\addplot+[mblue, thick, dashed, mark=square*, mark size=1.8pt, mark options={draw=mblue, fill=mblue}, forget plot]
coordinates {(250,17) (500,20) (750,22) (1000,25) (1250,25) (1500,27) (1750,27) (2000,27)};

\addplot+[mred, thick, solid, mark=*, mark size=1.8pt, mark options={draw=mred, fill=mred}]
coordinates {(250,19) (500,23) (750,25) (1000,27) (1250,29) (1500,29) (1750,30) (2000,30)};
\addlegendentry{$M$=500}

\addplot+[mred, thick, dashed, mark=square*, mark size=1.8pt, mark options={draw=mred, fill=mred}, forget plot]
coordinates {(250,19) (500,24) (750,26) (1000,29) (1250,29) (1500,30) (1750,30) (2000,30)};

\addplot+[mpink, thick, solid, mark=*, mark size=1.8pt, mark options={draw=mpink, fill=mpink}]
coordinates {(250,21) (500,26) (750,29) (1000,30) (1250,30) (1500,30) (1750,30) (2000,30)};
\addlegendentry{$M$=1000}

\addplot+[mpink, thick, dashed, mark=square*, mark size=1.8pt, mark options={draw=mpink, fill=mpink}, forget plot]
coordinates {(250,22) (500,27) (750,30) (1000,30) (1250,30) (1500,30) (1750,30) (2000,30)};

\addplot+[mcyan, thick, solid, mark=*, mark size=1.8pt, mark options={draw=mcyan, fill=mcyan}]
coordinates {(250,24) (500,30) (750,30) (1000,30) (1250,30) (1500,30) (1750,30) (2000,30)};
\addlegendentry{$M$=3000}

\addplot+[mcyan, thick, dashed, mark=square*, mark size=1.8pt, mark options={draw=mcyan, fill=mcyan}, forget plot]
coordinates {(250,26) (500,30) (750,30) (1000,30) (1250,30) (1500,30) (1750,30) (2000,30)};

\addplot+[black, very thick, solid, mark=*, mark size=1.8pt, mark options={draw=black, fill=black}]
coordinates {(250,2) (500,2) (750,3) (1000,4) (1250,3) (1500,5) (1750,4) (2000,5)};
\addlegendentry{\rlap{Identity covariance}}

\addplot+[black, very thick, dashed, mark=square*, mark size=1.8pt, mark options={draw=black, fill=black}, forget plot]
coordinates {(250,2) (500,2) (750,3) (1000,4) (1250,4) (1500,5) (1750,4) (2000,5)};



\end{axis}

\begin{axis}[
    width=\dimexpr\linewidth-1.15cm\relax,
    height=3.2cm,
    scale only axis,
    xmin=230, xmax=2050,
    ymin=0, ymax=31.5,
    axis x line=none,
    axis y line*=right,
    ytick={2,3,...,30},
    yticklabel={%
        \pgfmathtruncatemacro{\ytickint}{\tick}%
            \ifodd\ytickint
        \else
            \number\numexpr(\ytickint+2)*(\ytickint+3)/2\relax
        \fi
    },
    ylabel={Decision variables (r)},
    ylabel style={font=\small, yshift=3pt, rotate=180},
    tick label style={font=\scriptsize},
    tick style={black},
    axis line style={black},
    major grid style={draw=none},
    minor grid style={draw=none},
    clip=false,
]
\end{axis}

\pgfresetboundingbox
\path[use as bounding box]
    (-0.95cm,-0.35cm) rectangle (\dimexpr\linewidth-0.95cm\relax,4.08cm);
\end{tikzpicture}
    \end{minipage}
    \hfill
    \begin{minipage}[t]{0.56\textwidth}
        \centering
        \pgfplotstableread{
dx policy_lower policy_mean policy_upper value_lower value_mean value_upper
2 -0.000004 0.000039 0.000083 0.000342 0.000928 0.001513
3 -0.000013 0.000074 0.000160 0.000778 0.001259 0.001741
4 0.000053 0.000107 0.000161 0.001514 0.002318 0.003121
5 0.000135 0.000466 0.000796 0.002901 0.004805 0.006710
6 0.000108 0.000461 0.000814 0.002572 0.005123 0.007675
7 0.000349 0.000828 0.001306 0.004865 0.007863 0.010861
8 0.000208 0.000346 0.000483 0.003748 0.004958 0.006169
9 0.000444 0.000846 0.001248 0.004455 0.006798 0.009142
10 0.000711 0.001837 0.002963 0.004811 0.010728 0.016646
11 0.000477 0.000912 0.001347 0.005638 0.008564 0.011490
12 0.001419 0.002244 0.003069 0.008111 0.012182 0.016254
13 0.001196 0.001931 0.002666 0.009520 0.012292 0.015064
14 0.000944 0.001861 0.002778 0.007506 0.011786 0.016066
15 0.001306 0.002476 0.003647 0.008483 0.011773 0.015062
16 0.001470 0.003019 0.004568 0.009114 0.014204 0.019293
17 0.002129 0.004437 0.006745 0.011047 0.017295 0.023543
18 0.003305 0.004599 0.005894 0.012390 0.016841 0.021293
19 0.003669 0.005318 0.006966 0.011550 0.017812 0.024073
20 0.003831 0.005193 0.006555 0.010888 0.015223 0.019558
}\datafivehundred

\pgfplotstableread{
dx policy_lower policy_mean policy_upper value_lower value_mean value_upper
2 -0.000000 0.000004 0.000008 -0.000009 0.000293 0.000594
3 0.000001 0.000006 0.000011 0.000068 0.000252 0.000436
4 -0.000002 0.000010 0.000022 0.000144 0.000673 0.001203
5 -0.000010 0.000037 0.000084 0.000567 0.001442 0.002317
6 -0.000013 0.000076 0.000165 0.000543 0.002014 0.003484
7 0.000042 0.000112 0.000183 0.001354 0.003118 0.004882
8 0.000014 0.000073 0.000132 0.001173 0.002183 0.003193
9 0.000015 0.000149 0.000283 0.001518 0.002807 0.004097
10 -0.000025 0.000279 0.000584 0.001404 0.004582 0.007759
11 0.000045 0.000267 0.000489 0.002087 0.004471 0.006854
12 0.000105 0.000448 0.000791 0.003538 0.006112 0.008686
13 0.000214 0.000395 0.000575 0.004587 0.006539 0.008491
14 0.000125 0.000486 0.000847 0.003826 0.006769 0.009713
15 0.000389 0.000682 0.000976 0.005012 0.007322 0.009631
16 0.000365 0.000619 0.000873 0.004968 0.008181 0.011395
17 0.000515 0.001171 0.001828 0.006331 0.011102 0.015873
18 0.000532 0.000995 0.001458 0.006852 0.009748 0.012644
19 0.000874 0.001814 0.002754 0.007271 0.012279 0.017286
20 0.000494 0.001403 0.002313 0.006616 0.009811 0.013005
21 0.001473 0.002375 0.003278 0.008041 0.012073 0.016106
22 0.000882 0.002337 0.003792 0.006702 0.011818 0.016935
23 0.000900 0.002901 0.004903 0.007625 0.014722 0.021818
24 0.001909 0.003421 0.004934 0.011107 0.015923 0.020739
25 0.002310 0.003870 0.005430 0.009961 0.014430 0.018899
26 0.003827 0.005297 0.006767 0.013405 0.016290 0.019175
27 0.005223 0.009715 0.014206 0.011045 0.013934 0.016824
}\datatwothousand

\begin{tikzpicture}
\begin{groupplot}[
    group style={group size=2 by 1, horizontal sep=0.1cm},
    width=0.425\linewidth,
    height=3.2cm,
    scale only axis,
    xmin=2, xmax=35.5,
    ymin=-0.0011, ymax=0.0253,
    xtick={5,10,15,20,25,30,35},
    ytick={0,0.005,0.010,0.015,0.020,0.025},
    scaled ticks=false,
    tick label style={font=\scriptsize},
    xticklabel={\pgfmathprintnumber[fixed,precision=0]{\tick}},
    yticklabel={\pgfmathprintnumber[fixed,fixed zerofill,precision=3]{\tick}},
    label style={font=\scriptsize},
    title style={font=\small, yshift=-1mm},
    grid=major,
    grid style={draw=gridgray, line width=0.28pt, opacity=0.65},
    axis line style={draw=black!55, line width=0.5pt},
    tick style={draw=black!55, line width=0.5pt},
    axis background/.style={fill=white},
    axis on top=true,
    clip=true,
    enlargelimits=false,
    legend style={
        draw=black!13,
        fill=white,
        fill opacity=0.90,
        text opacity=1,
        font=\scriptsize,
        cells={anchor=west},
        inner xsep=1pt,
        inner ysep=1pt,
        row sep=1pt
    },
    legend image code/.code={
        \draw[
            #1,
            mark repeat=2,
            mark phase=2,
            mark options={solid}
        ]
        plot coordinates {
            (0cm,0cm)
            (0.13cm,0cm)
            (0.26cm,0cm)
        };
    },
    online/.style={
        onlineblue,
        thick,
        mark=*,
        mark size=1.1pt,
        mark options={solid, fill=onlineblue, draw=onlineblue}
    },
    value/.style={
        valueorange,
        thick,
        mark=square*,
        mark size=1.1pt,
        mark options={solid, fill=valueorange, draw=valueorange}
    },
]

\nextgroupplot[
    legend style={at={(0.018,0.962)}, anchor=north west}
]
\path[fill=unboundedgray] (axis cs:20.35,-0.0011) rectangle (axis cs:35.5,0.0253);

\addplot[name path=p500u, draw=none, forget plot] table[x=dx,y=policy_upper] {\datafivehundred};
\addplot[name path=p500l, draw=none, forget plot] table[x=dx,y=policy_lower] {\datafivehundred};
\addplot[onlineblue, fill opacity=0.21, draw=none, forget plot] fill between[of=p500u and p500l];

\addplot[name path=v500u, draw=none, forget plot] table[x=dx,y=value_upper] {\datafivehundred};
\addplot[name path=v500l, draw=none, forget plot] table[x=dx,y=value_lower] {\datafivehundred};
\addplot[valueorange, fill opacity=0.20, draw=none, forget plot] fill between[of=v500u and v500l];

\addplot[online] table[x=dx,y=policy_mean] {\datafivehundred};
\addlegendentry{online performance}
\addplot[value] table[x=dx,y=value_mean] {\datafivehundred};
\addlegendentry{value function}

\node[rotate=50, align=center, font=\scriptsize] at (axis cs:27.3,0.0118) {LP\\unbounded};
\node[anchor=north east, draw=black!13, fill=white, fill opacity=0.82, text opacity=1, inner xsep=4pt, inner ysep=2pt, font=\scriptsize
] at (rel axis cs:0.965,0.94) {$N=500$};

\nextgroupplot[
    yticklabels={,,, , ,},
    legend style={at={(0.018,0.962)}, anchor=north west}
]
\path[fill=unboundedgray] (axis cs:27.5,-0.0011) rectangle (axis cs:35.5,0.0253);

\addplot[name path=p2000u, draw=none, forget plot] table[x=dx,y=policy_upper] {\datatwothousand};
\addplot[name path=p2000l, draw=none, forget plot] table[x=dx,y=policy_lower] {\datatwothousand};
\addplot[onlineblue, fill opacity=0.21, draw=none, forget plot] fill between[of=p2000u and p2000l];

\addplot[name path=v2000u, draw=none, forget plot] table[x=dx,y=value_upper] {\datatwothousand};
\addplot[name path=v2000l, draw=none, forget plot] table[x=dx,y=value_lower] {\datatwothousand};
\addplot[valueorange, fill opacity=0.20, draw=none, forget plot] fill between[of=v2000u and v2000l];

\addplot[online] table[x=dx,y=policy_mean] {\datatwothousand};
\addlegendentry{online performance}
\addplot[value] table[x=dx,y=value_mean] {\datatwothousand};
\addlegendentry{value function}

\node[rotate=50, align=center, font=\scriptsize] at (axis cs:31.6,0.0115) {LP\\unbounded};
\node[anchor=north east, draw=black!13, fill=white, fill opacity=0.82, text opacity=1, inner xsep=4pt, inner ysep=2pt, font=\scriptsize
] at (rel axis cs:0.965,0.94) {$N=2000$};

\end{groupplot}

\node[font=\bfseries\small]
    at ($(group c1r1.north)!0.5!(group c2r1.north)+(0,0.25cm)$)
    {Moment Matching Performance};

\node[rotate=90, font=\small]
    at ($(group c1r1.north west)!0.5!(group c1r1.south west)+(-1.05cm,0)$)
    {Normalized difference};

\node[font=\small]
    at ($(group c1r1.south)!0.5!(group c2r1.south)+(0,-0.55cm)$)
    {State dimension (n)};

\pgfresetboundingbox
\path[use as bounding box]
    (-0.95cm,-0.35cm) rectangle (\dimexpr\linewidth-0.95cm\relax,4.08cm);
\end{tikzpicture}
        
    \end{minipage}
    \caption{{\review \textbf{LTI Systems.} On the left, we show the maximum tested state dimension for which condition \eqref{ifcond} is feasible in at least $50\%$ (dashed lines) and $100\%$ (solid lines) of instances, versus the number $N$ of samples, and for different auxiliary pool sizes $M$. Larger $M$ improves boundedness, while the fixed objective baseline remains bounded only in low dimension. On the right, normalized value-function and closed-loop cost errors relative to the analytic LQR solution, for $N=500$ and $N=2000$. Solid lines show the mean across seeds and shaded regions denote one standard deviation. $M$ is set to $250$.}
    }
\label{fig:lti_systems}
\end{figure*}
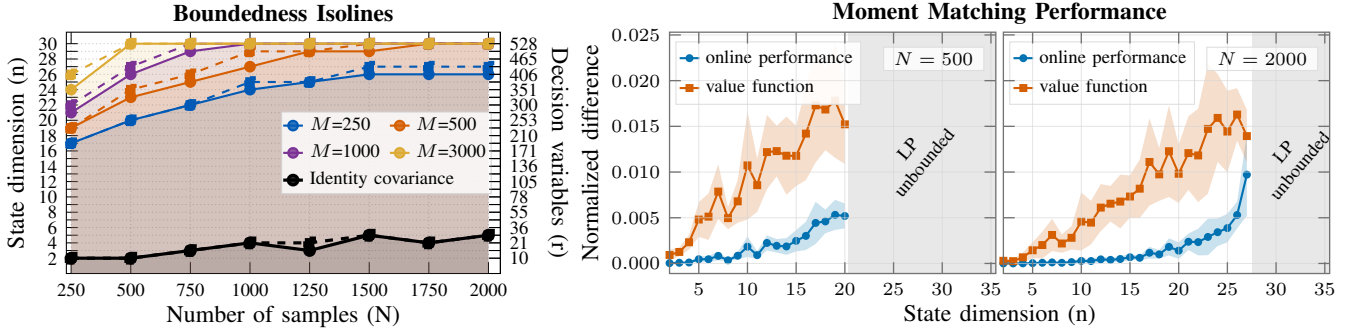

We illustrate the proposed moment-matching LP method on high-dimensional, minimal-data, model-free optimal control problems. We compare our method against the {\review standard} LP method with a fixed cost vector. The experiments not only demonstrate the effectiveness of the moment-matching for obtaining bounded solutions, but also for convincing closed-loop performance of the learned policies.
Code and parameter choices are available at \href{https://github.com/lucia-pezzetti/Bounded-One-Shot-Linear-Programming}{https://github.com/lucia-pezzetti/Bounded-One-Shot-Linear-Programming}.

\subsection{LTI Systems}

We consider LTI systems $x_{k+1} = A x_k + B u_k$ of increasing state dimension $n$, input of dimension 2, and quadratic stage cost $\ell(x,u)=\|x\|^2+0.1\|u\|^2$. 
For each state dimension, we randomly generate controllable matrices $A,B$ and collect datasets of size $N$ by uniformly sampling state–action pairs and computing the corresponding next states and stage costs, populating the matrices $\bold \Phi$ and $\boldsymbol \ell$ in \eqref{data-driven_LP}. The LP is solved by selecting $q(x,u)= \begin{bsmallmatrix}
    x \\ u
\end{bsmallmatrix}^\top Q \begin{bsmallmatrix}
    x \\ u
\end{bsmallmatrix}\in\Sc^{\text{quad}}$, which produces linear feedback policies when a bounded solution exists. We compare the proposed moment-matching method with $M=250, 500, 1000, 3000$ auxiliary samples $y$ from a uniform distribution with a baseline LP with $c=\mathcal{N}(0,I)$.

Fig.~\ref{fig:lti_systems} reports {\review the largest tested state dimension for which condition~\eqref{ifcond} is feasible in at least $50\%$ (dashed lines) and $100\%$ (solid lines) of random instances, for different values of $N$ and $M$. Whenever \eqref{ifcond} is feasible, Proposition~\ref{proppoly} guarantees that the constructed objective yields a finite LP optimal value. The area under the dashed (resp. solid) line corresponds to combinations of state dimensions and constraints for which the LP is bounded in at least $50\%$ (resp. $100\%$) of instances. Larger $M$ gives a richer
auxiliary representation in~\eqref{data-driven_LP}, making the moment-matching equality easier to satisfy. The fixed-covariance baseline is independent of $M$ and remains bounded only in low dimension. For example, with $M=3000$, moment-matching reaches $100\%$ boundedness up
to $n=30$ by $N=500$, whereas the fixed-covariance baseline reaches only $n=2$.} {\review Note that in most scenarios where the fixed cost vector leads to an unbounded solution (hence the feasible region is unbounded) our method finds a finite solution instead.}

The right-hand side of Fig.~\ref{fig:lti_systems} shows that {\review when the moment-matching LP has a finite optimum, the learned policies are close to the analytic LQR
solution.} 
We report i) the normalized difference between the optimal value function {\review (computed by solving the Riccati equation)} and the LP solution, and ii) the normalized difference between the optimal value function and the \textit{online performance} of the greedy policy. 
The learned policies achieve performance within approximately $1\%$ of optimal. Increasing the number of samples reduces errors, indicating convergence toward the optimal quadratic solution.

\subsection{Nonlinear Mechanical Systems}

\begin{figure*}[!t]
    \centering
    \begin{minipage}[t]{0.38\textwidth}
    \centering
    \begin{tikzpicture}
\begin{axis}[
    width=\dimexpr\linewidth-1.15cm\relax,
    height=3.2cm,
    scale only axis,
    xmin=230, xmax=10250,
    ymin=0, ymax=10.5,
    scaled x ticks=false,
    xlabel={Number of samples (N)},
    ylabel={State dimension (n)},
    title={\textbf{Boundedness Isolines}},
    xlabel style={font=\small, yshift=3pt},
    ylabel style={font=\small, yshift=-3pt},
    title style={font=\small, yshift=-5.3pt},
    tick label style={font=\scriptsize},
    xtick={1000,2000,3000,4000,5000,6000,7000,8000,9000,10000},
    xticklabels={1000,2000,3000,4000,5000,6000,7000,8000,9000,10000},
    ytick={2,4,6,8,10},
    yticklabels={2,4,6,8,10},
    minor x tick num=0,
    axis line style={black},
    tick style={black},
    grid=both,
    major grid style={gray!45, densely dotted, line width=0.15pt},
    minor grid style={gray!25, densely dotted, line width=0.15pt},
    legend style={
        at={(0.98,0.5)},
        anchor=east,
        draw=gray!35,
        fill=white,
        fill opacity=0.82,
        text opacity=1,
        font=\scriptsize,
        inner sep=0pt,
        legend cell align=left,
        /tikz/every even column/.append style={column sep=0.05cm},
    },
    legend columns=2,
    legend image code/.code={
        \draw[
            #1,
            mark repeat=2,
            mark phase=2,
            mark options={solid}
        ]
        plot coordinates {
            (0cm,0cm)
            (0.13cm,0cm)
            (0.26cm,0cm)
        };
    },
    clip mode=individual,
]

\addplot[draw=none, fill=mplblue, fill opacity=0.12, forget plot]
coordinates {
    (1000,0) (1000,6) (2000,6) (3000,8) (4000,8) (5000,8)
    (6000,8) (7000,8) (8000,8) (9000,8) (10000,8)
    (10000,0)
} -- cycle;

\addplot[draw=none, fill=mplblue, fill opacity=0.12, forget plot]
coordinates {
    (1000,0) (1000,6) (2000,6) (3000,6) (4000,8) (5000,8)
    (6000,8) (7000,8) (8000,8) (9000,8) (10000,8)
    (10000,0)
} -- cycle;

\addplot[draw=none, fill=mplred, fill opacity=0.12, forget plot]
coordinates {
    (1000,0) (1000,6) (2000,8) (3000,8) (4000,8) (5000,8)
    (6000,8) (7000,10) (8000,10) (9000,10) (10000,10)
    (10000,0)
} -- cycle;

\addplot[draw=none, fill=mplred, fill opacity=0.12, forget plot]
coordinates {
    (1000,0) (1000,6) (2000,6) (3000,8) (4000,8) (5000,8)
    (6000,8) (7000,8) (8000,8) (9000,8) (10000,10)
    (10000,0)
} -- cycle;

\addplot[draw=none, fill=mplpink, fill opacity=0.12, forget plot]
coordinates {
    (1000,0) (1000,6) (2000,8) (3000,8) (4000,10) (5000,10)
    (6000,10) (7000,10) (8000,10) (9000,10) (10000,10)
    (10000,0)
} -- cycle;

\addplot[draw=none, fill=mplpink, fill opacity=0.12, forget plot]
coordinates {
    (1000,0) (1000,6) (2000,8) (3000,8) (4000,8) (5000,10)
    (6000,10) (7000,10) (8000,10) (9000,10) (10000,10)
    (10000,0)
} -- cycle;

\addplot[draw=none, fill=mplcyan, fill opacity=0.12, forget plot]
coordinates {
    (1000,0) (1000,8) (2000,8) (3000,10) (4000,10) (5000,10)
    (6000,10) (7000,10) (8000,10) (9000,10) (10000,10)
    (10000,0)
} -- cycle;

\addplot[draw=none, fill=mplcyan, fill opacity=0.12, forget plot]
coordinates {
    (1000,0) (1000,8) (2000,8) (3000,8) (4000,8) (5000,10)
    (6000,10) (7000,10) (8000,10) (9000,10) (10000,10)
    (10000,0)
} -- cycle;

\addplot[draw=none, fill=mplgray, fill opacity=0.12, forget plot]
coordinates {
    (1000,0) (1000,2) (2000,2) (3000,2) (4000,2) (5000,2)
    (6000,2) (7000,2) (8000,2) (9000,2) (10000,2)
    (10000,0)
} -- cycle;


\addplot+[mplblue, thick, dashed, mark=square*, mark size=1.8pt,
    mark options={draw=mplblue, fill=mplblue}, forget plot]
coordinates {
    (1000,6) (2000,6) (3000,8) (4000,8) (5000,8)
    (6000,8) (7000,8) (8000,8) (9000,8) (10000,8)
};

\addplot+[mplblue, thick, solid, mark=*, mark size=1.8pt,
    mark options={draw=mplblue, fill=mplblue}]
coordinates {
    (1000,6) (2000,6) (3000,6) (4000,8) (5000,8)
    (6000,8) (7000,8) (8000,8) (9000,8) (10000,8)
};
\addlegendentry{$M$=500}

\addplot+[mplred, thick, dashed, mark=square*, mark size=1.8pt,
    mark options={draw=mplred, fill=mplred}, forget plot]
coordinates {
    (1000,6) (2000,8) (3000,8) (4000,8) (5000,8)
    (6000,8) (7000,10) (8000,10) (9000,10) (10000,10)
};

\addplot+[mplred, thick, solid, mark=*, mark size=1.8pt,
    mark options={draw=mplred, fill=mplred}]
coordinates {
    (1000,6) (2000,6) (3000,8) (4000,8) (5000,8)
    (6000,8) (7000,8) (8000,8) (9000,8) (10000,10)
};
\addlegendentry{$M$=1000}

\addplot+[mplpink, thick, dashed, mark=square*, mark size=1.8pt,
    mark options={draw=mplpink, fill=mplpink}, forget plot]
coordinates {
    (1000,6) (2000,8) (3000,8) (4000,10) (5000,10)
    (6000,10) (7000,10) (8000,10) (9000,10) (10000,10)
};

\addplot+[mplpink, thick, solid, mark=*, mark size=1.8pt,
    mark options={draw=mplpink, fill=mplpink}]
coordinates {
    (1000,6) (2000,8) (3000,8) (4000,8) (5000,10)
    (6000,10) (7000,10) (8000,10) (9000,10) (10000,10)
};
\addlegendentry{$M$=2500}

\addplot+[mplcyan, thick, dashed, mark=square*, mark size=1.8pt,
    mark options={draw=mplcyan, fill=mplcyan}, forget plot]
coordinates {
    (1000,8) (2000,8) (3000,10) (4000,10) (5000,10)
    (6000,10) (7000,10) (8000,10) (9000,10) (10000,10)
};

\addplot+[mplcyan, thick, solid, mark=*, mark size=1.8pt,
    mark options={draw=mplcyan, fill=mplcyan}]
coordinates {
    (1000,8) (2000,8) (3000,8) (4000,8) (5000,10)
    (6000,10) (7000,10) (8000,10) (9000,10) (10000,10)
};
\addlegendentry{$M$=5000}

\addplot+[mplgray, very thick, dashed, mark=square*, mark size=1.8pt,
    mark options={draw=mplgray, fill=mplgray}, forget plot]
coordinates {
    (1000,2) (2000,2) (3000,2) (4000,2) (5000,2)
    (6000,2) (7000,2) (8000,2) (9000,2) (10000,2)
};

\addplot+[mplgray, very thick, solid, mark=*, mark size=1.8pt,
    mark options={draw=mplgray, fill=mplgray}]
coordinates {
    (1000,2) (2000,2) (3000,2) (4000,2) (5000,2)
    (6000,2) (7000,2) (8000,2) (9000,2) (10000,2)
};
\addlegendentry{\rlap{Identity covariance}}

\end{axis}

\begin{axis}[
    width=\dimexpr\linewidth-1.15cm\relax,
    height=3.2cm,
    scale only axis,
    xmin=230, xmax=10250,
    ymin=0, ymax=10.5,
    axis x line=none,
    axis y line*=right,
    ytick={2,4,6,8,10},
    yticklabels={21,120,406,1035,2211},
    ylabel={Decision variables (r)},
    ylabel style={font=\small, yshift=6pt, rotate=180},
    tick label style={font=\scriptsize},
    tick style={black},
    axis line style={black},
    major grid style={draw=none},
    minor grid style={draw=none},
    clip=false,
]
\end{axis}

\pgfresetboundingbox
\path[use as bounding box]
    (-0.95cm,-0.35cm) rectangle (\dimexpr\linewidth-0.95cm\relax,4.08cm);
\end{tikzpicture}
    \end{minipage}
    \hfill
    \begin{minipage}[t]{0.6\textwidth}
        \centering
        \input{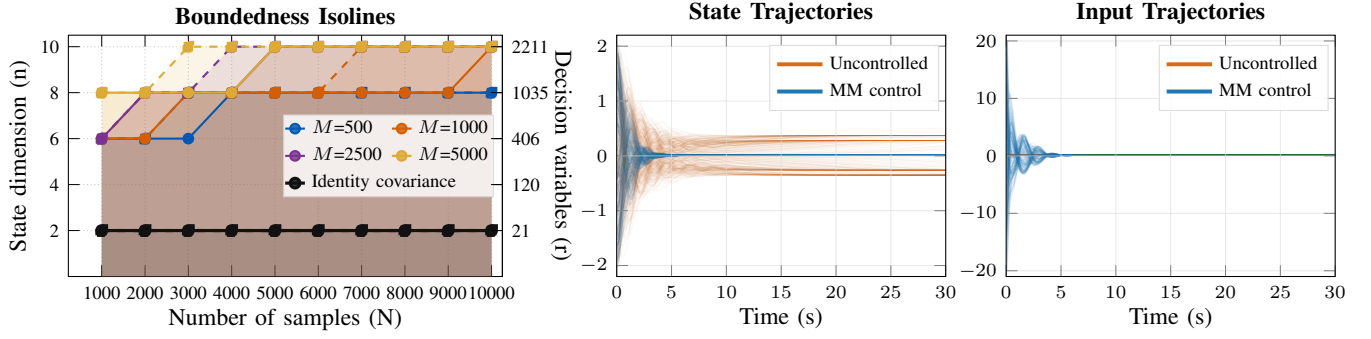}
        
    \end{minipage}
    \caption{{\review \textbf{Nonlinear Mechanical Systems.} On the left, we show the largest tested state dimension for which condition \eqref{ifcond} is feasible in at least $50\%$ (dashed lines) and $100\%$ (solid lines) of instances, for different auxiliary pool sizes $M$. Also in this nonlinear example, the moment-matching formulation remains feasible for significantly larger dimensions compared to the baseline and larger $M$ improves boundedness. Note that, because of the polynomial features, here the number of decision variables scales combinatorially with the state dimension. On the right, the moment-matching (MM) controller steers the system to the unstable equilibrium at the origin, while the uncontrolled trajectories are attracted by other stable equilibria.}}
    \label{fig:nonlinear_system}
\end{figure*}

We next consider nonlinear $n$-dimensional point-mass systems of increasing dimension with elastic coupling, {\review nonlinear gravitational term,} and cubic velocity drag. 
The continuous-time dynamics are
\begin{equation*}
\dot p = v, 
\qquad
m \dot v = -Kp + {\review G\tanh{p}} -({\review\beta}\|v\|^2 + {\review\xi}) v + Bu,
\end{equation*}
where the state $x = \begin{bmatrix}
    p & v
\end{bmatrix}^\top \in \mathbb{R}^{n}$ collects the positions and velocities of the masses, $u \in \mathbb{R}$ is a scalar input, ${\review\beta}\in\mathbb{R}_+$ is the cubic-drag coefficient, {\review $\xi \in\mathbb{R}_+$ is the viscous damping coefficient, $K,G\in\mathbb{R}^{n/2 \times n/2}$ are the stiffness and gravitational matrices, respectively. All the masses are equal to $m\in\mathbb{R}_+$. The linearized dynamics at the origin are $\dot p = v$ and $m \dot v = (G-K)p-\xi v$. At each instance, we select $K$ and $G$ so that $G-K$ has at least one positive eigenvalue, making the origin unstable.} The system is discretized with {\review a fourth-order Runge-Kutta scheme} and a discounted infinite-horizon objective with stage cost $\ell(x,u)=\|x\|^2+0.01\|u\|^2+{\review\sum_i p_i^4}$ is considered. To assess robustness across instances, physical parameters are randomized across seeds. The approximate LP~\eqref{data-driven_LP} is solved using $q\in\Sc^{\text{poly}}_{u^2}$ with degree 4 in $x$.
We compare the proposed moment-matching method of Proposition~\ref{proppoly}, with $M=500, 1000, 2500, 5000$ uniformly distributed auxiliary samples, with a baseline LP using $c=\mathcal{N}(0,I)$.

    Fig.~\ref{fig:nonlinear_system} {\review reports the largest state dimension for which condition~\eqref{ifcond} is feasible in at least $50\%$ and $100\%$ of the runs, for different auxiliary pool sizes $M$. The nonlinear result obtained can be compared with the linear one in Fig.~\ref{fig:lti_systems} by analyzing how the boundedness rate scales with the number of decision variables in the LP, rather than the state dimension. Similarly to the linear case, the moment-matching formulation remains feasible for significantly larger dimensions and smaller datasets than the baseline and increasing $M$ improves boundedness. Fig.~\ref{fig:nonlinear_system} (right) illustrates closed-loop trajectories for $n=4$, showing that the learned polynomial controllers steer the system toward the unstable equilibrium.}

\subsection{Computational Aspects}

Conditions \eqref{iffcond}-\eqref{ifcond} are computationally tractable: the former is a linear matrix inequality in the variable $\lambda$, and the latter a linear feasibility problem with $r$ linear equations and nonnegativity constraints in the variables $\lambda$ and $\mu$. However, since typically $N+M>r$, this may result in a set of underdetermined equations, suggesting that there exist many directions leading to a finite value. We set $\|\lambda\|_1 = 1$ to fix the scale and to avoid the zero solution; otherwise, the solver-selected solution is arbitrary. The problem of searching for specific directions to maximize performance is worth investigating and deferred to future studies.

\section{CONCLUSIONS}

The proposed moment-matching method computes a data-driven cost vector leading to a finite solution for the approximate LP approach. In addition to the usual cost of solving the LP, our method adds the cost of a set of linear equations, but, in return, finally positions the LP approach as a reliable alternative to other dynamic programming formulations.

We provide results for deterministic nonlinear systems and polynomial features, but we are interested in extending the method to stochastic dynamics and richer basis functions. Finally, although the experiments are encouraging, we envision further investigation of theoretical performance bounds.



\bibliographystyle{IEEEtran}
\bibliography{Bibliography}

\end{document}